\documentclass[12pt]{amsart}
\usepackage{hyperref}
\usepackage{amsmath, amsthm,amstext}
\usepackage{amssymb, array, amsfonts}
\usepackage[english]{babel}
\usepackage{url}
\usepackage{amsrefs}
\usepackage{float}
\usepackage{graphics}
\usepackage{graphicx}
   \usepackage{enumerate}
\numberwithin{equation}{section}

\usepackage[margin=1.2in]{geometry}

\def \F {\mathcal F}

\renewcommand{\l}{\left}
\renewcommand{\r}{\right}

\def \C{\mathbb{C}}
\def \mnc{\mathbf{M}_n(\mathbb C)}
\def \Q{\mathbb{Q}}

\def \N{\mathbb{N}}
\def \M2{\mathrm{M}_2}
\def \R{\mathbb{R}}
\def \Z{\mathbb{Z}}
\def \T{\mathbb{T}}

\def \sl2r{\mathrm{SL}(2,\R)}

\newcommand{\beq}{\begin{equation}}
\newcommand{\eeq}{\end{equation}}

\def\ran{\operatorname{ran}}

\def \one{\mathbf{1}}

\newcommand{\eqdef}{\stackrel{\rm def}{=\kern-3.6pt=}}
	
\renewcommand{\>}{\rangle}

\theoremstyle{plain}
\newtheorem{theorem}{\bf Theorem}[section]
\newtheorem{lemma}[theorem]{\bf Lemma}

\newtheorem{prop}[theorem]{\bf Proposition}
\newtheorem{cor}[theorem]{\bf Corollary}

\theoremstyle{definition}

\theoremstyle{remark}
\newtheorem{remark}[theorem]{\bf Remark}

\renewcommand{\le}{\leqslant}
\renewcommand{\ge}{\geqslant}

\newcommand{\dist}{\mathop{\mathrm{dist}}\nolimits}

\newcommand{\dc}{\mathop{\mathrm{DC}}\nolimits}
\newcommand{\dcct}{\mathop{\mathrm{DC(c,\tau)}}\nolimits}

\renewcommand{\qed}{\vrule height7pt width5pt depth0pt}
\title[Unbounded monotone potentials]{Localization for quasiperiodic operators with unbounded monotone potentials}
\author[I. Kachkovskiy]{Ilya Kachkovskiy}
\address{Department of Mathematics,
	Michigan State University,
	Wells Hall, 619 Red Cedar Road,
	East Lansing, MI 48824,
	United States of America}
\email{ikachkov@msu.edu}

\date{}
\begin{document}

\begin{abstract}
	We establish non-perturbative Anderson localization for a wide class of 1D quasiperiodic operators with unbounded monotone potentials, extending the classical results on Maryland model, and perturbative results for analytic potentials by B\'ellissard, Lima, and Scoppola.
	\end{abstract}
\maketitle
\section{Introduction and main results}
In this paper, we will consider the following class of quasiperiodic operators on $\ell^2(\mathbb Z)$:
\beq
\label{h_def}
(H(x)\psi)(n)=\psi(n+1)+\psi(n-1)+f(x+n\alpha)\psi(n),\quad x\in \mathbb R\setminus(\mathbb Z+\alpha\mathbb Z).
\eeq
We will assume that $f\in \mathcal F(\gamma)$, which, by definition, will mean
\begin{enumerate}
	\item[$(\F_1)$] $f$ is defined on $\R\setminus\Z$, is $1$-periodic, continuous on $(0,1)$, and $f(0+)=-\infty$, $f(1-)=+\infty$.
	\item[$(\F_2)$] is Lipschitz monotone. That is, there exists $\gamma>0$ such that $f(y)-f(x)\ge \gamma(y-x)$ for all $0<x<y<1$.
	\item[$(\F_3)$] $\log |f|\in {L}^1(0,1)$.
\end{enumerate}

The study of this class of functions is motivated by the Maryland model, which is a special case $f(x)=\lambda \tan(\pi x)$. Classical results \cites{GFPS,Pas1,CFKS,S1} show that Maryland model has Anderson localization for Diophantine frequencies $\alpha$ and all coupling constants $\lambda\neq 0$. Recently, a complete spectral description of this model was obtained in \cite{JL1}. While the existing techniques for Maryland model give extremely sharp results, they are based either on the analysis of an explicit cohomological equation, or (see \cite{JY2}) on fine structure of trigonometric polynomials in the spirit of \cite{J_annals}, in both cases relying on the specific form of $f$.

Thus, a natural question arises: is there a proof of localization that only uses qualitative properties of $f$, such as monotonicity, and does not refer to the specific trigonometric structure? A partial answer to this question was given in \cite{Bel1}. Using a KAM-type procedure, the authors obtained Anderson localization for a class of meromorphic functions $f$ whose restrictions onto $\mathbb R$ are also $1$-periodic and Lipschitz monotone. Their argument (as well as the classical Maryland model) extends to the potentials on $\Z^d$. However, due to the nature of the KAM technique, the result is perturbative. That is, once $f$ is fixed, one can only obtain localization for the potential $\lambda f$ with $\lambda\ge \lambda_0(\alpha)$, where $\lambda_0$ depends on the Diophantine constant of $\alpha$ and does not have a uniform lower bound on a full measure set of frequencies. One should also mention \cite{Pas2}, where the authors obtain a localization result in a different context (on a half-line, with randomized boundary condition at the origin), and results on singular continuous spectrum \cite{JY,Shiwen}.

On the other hand, Maryland model is known, in case of Diophantine frequencies, to demonstrate non-perturbative localization for all $\lambda\neq 0$. Thus, a natural question would be whether or not general monotone $f$ demonstrate localization at small coupling. For the case of {\it bounded} Lipschitz monotone $f$, small coupling localization was obtained in \cite{JK}. It is natural to expect that it should be easier for unbounded potentials to generate pure point spectrum. However, the argument of \cite{JK} significantly relies on boundedness of $f$.

As usual, we call an irrational number $\alpha$ Diophantine (denoted $\alpha\in \dc(c,\tau)$ for some $c,\tau>0$) if
\beq
\label{dc_def}
\dist(k\alpha,\mathbb Z)\ge c|k|^{-\tau},\quad \forall k\in \Z\setminus\{0\}.
\eeq
We also use the notation $\dc=\cup_{c>0,\tau\ge 1 }\dc(c,\tau)$.
\begin{theorem}
\label{main}
Suppose $\alpha\in \dc$, $f\in \mathcal{F(\gamma)}$ for some $\gamma>0$. Then, for a.~e. $x$, the spectrum of the operator \eqref{h_def} is purely point, and the eigenfunctions decay exponentially. In addition, if $\gamma>2e$, then the statement holds for all $x\in \R\setminus(\Z+\alpha\Z)$.
\end{theorem}
We also show that the spectrum of $H$ is the real line.
\begin{theorem}
\label{th_spec_R}
Suppose $f$ satisfies $(\F_1)$. Then $\sigma(H(x))=\R$ for all $\alpha\in \R$, $x\in \R\setminus(\Z+\alpha\Z)$.
\end{theorem}

\begin{remark}
Condition $(\F_3)$ is necessary for the Lyapunov exponent of $H$ to be finite, and is required for several techniques involved in the proof of Theorem \ref{main}, including Thouless formula, Kotani theory, and large deviation theorems.
\end{remark}
\begin{remark}
In Condition $(\F_1)$, one can assign any values $-\infty\le f(0+)<f(1-)\le +\infty$ without significant changes in the proof. In particular, while we may lose continuity of the eigenvalue curves $\Lambda_j$, the expression \eqref{lambdaj_def} still gives correct Lipschitz monotone parametrization of the box eigenvalues. In the bounded case, it shows that one can {\it drop the upper Lipschitz requirement in \cite{JK}}. Moreover, one can allow discontinuous $f$ as long as there are no resonances between the discontinuity points (no distance between two such points should belong to $\Z+\alpha \Z$).
\end{remark}
\begin{remark}
We would like to emphasize that monotonicity is not required for Theorem \ref{th_spec_R}. This shows that unbounded potentials behave differently from the bounded case, where different directions of monotonicity can create resonances and spectral gaps. We believe that the reason is that, generically, any horizontal line intersects the graph of a function satisfying $\F_1$ odd number of times.
\end{remark}
\begin{remark}
The results and methods of the present paper, as well as \cite{JK}, establish small coupling localization, which also holds for one-dimensional Anderson model and is much less typical for quasiperiodic operators. One can draw some parallels not only in the results, but also in proofs of localization. For example, Theorem \ref{main_ids} provides the same statement as Wegner's bound for a random potential with distribution function $f^{-1}$ (with the same Lipschitz constant), and Proposition \ref{loc_sufficient} is a version of spectral averaging.
\end{remark}
{\noindent \bf Structure of the paper. }Section 2 contains main definitions from spectral theory of ergodic operators, with some references specific for the unbounded case. In Section 3, we discuss upper bounds on transfer matrices. We show that an analogue of Furman's theorem holds modulo an unbounded ``tail'' which has explicit dependence on $f$. In Section 4, we show that Schr\"odinger operators with Lipschitz monotone potentials always have Lipschitz continuous IDS. The new proof is shorter and holds in more general context than the one provided in \cite{JK}. In particular, it does not use upper Lipschitz bound and provides a natural Lipschitz constant predicted by Wegner's bound.
In section 5, we establish more delicate claims about box eigenvalue counting function which is an ingredient for the large deviation theorem. We show that, if one compares box eigenvalue counting functions for two values of $x$, they can differ at most by an absolute constant, independent of the size of the box.
In Section 6, we use that observation to establish a large deviation theorem: namely, we identify the contribution of box eigenvalues that cause $\frac{1}{n}\log|\det (H_n(x)-E)|$ to deviate from it average value $L(E)$. The contribution from eigenvalues close to $E$ forms a well controlled large deviation set. The contribution from very large eigenvalues forms a ``tail''. The key observation is the following: in the Cramer's representation of Green's function through a ratio of determinants, the tail in the denominator always cancels the tail in the numerator. Therefore, one does not need to worry about abnormally large eigenvalues. After all preparations, the proof of Theorem \ref{main} is completed in Section 7 in the standard way. 

Section 8 is independent from Sections 3 -- 7 and contains the proof of Theorem \ref{th_spec_R}. The argument is relatively straightforward for $\alpha\in \Q$. The main difficulty is rational approximation: the standard proofs for the bounded case \cite{avron_simon_ids,bellissard_simon_cantor} require compactness of the set of phases on $\T^1$ for which $H(x)$ is defined. In our case, we have to avoid $x\in \Z+\alpha \Z$. To overcome this difficulty, we remove the lattice site on which the potential becomes unbounded, and consider the operator restricted to a smaller subspace. While we cannot claim that these restricted operators have the same spectrum as the rest of the ergodic family, their properties are sufficient to obtain a contradiction if one assumes that the original family has a gap.

The parameter $\gamma$ plays the role of a coupling constant. It is interesting to compare the regimes of small coupling with \cite{JK}. In both cases, we have a simple lower bound on the Lyapunov exponent for $\gamma>2e$ (large coupling), but the case of small coupling is treated by an implicit argument that guarantees that $L(E)>0$ for Lebesgue almost every energy. In \cite{JK}, the reason is that discontinuous $f$ lead to non-deterministic potentials \cite{DamanikKillip}, and in the present paper we rely on a spectral result \cite{SS} and Kotani theory. It is not clear whether \cite{DamanikKillip} can be extended to the unbounded case, or if there is an independent proof of uniform positivity of $L(E)$. In both cases, small coupling still demonstrates localization for almost every $x$, due to spectral averaging.

\section{Preliminaries from spectral theory}
In this section, we formulate some basic results from spectral theory and dynamics, and establish immediate consequences for the model \eqref{h_def}. Let 
$$
H_n(x)=\l.\one_{[0,n-1]}H(x)\r|_{\ran \one_{[0,n-1]}}
$$
be the $(n\times n)$-block of $H(x)$, or equivalently the Dirichlet restriction of the operator $H(x)$ onto $[0,n-1]\cap \Z$. Denote by $N_n(x,E)$ the counting function of the eigenvalues of $H_n(x)$ (defined for $x\in[0,1)$ except for finitely many points):
$$
N_n(x,E)=\#\sigma(H_n(x))\cap (-\infty,E].
$$
The {\it integrated density of states} of the operator family $H(x)$ is defined as
\beq
\label{ids_def}
N(E)=\lim\limits_{n\to \infty}\frac{1}{n}\int_{[0,1)}N_n(x,E)\,dx.
\eeq
Note that one can use any boundary conditions in the definition of $N_n(x)$ (for example, replace the Dirichlet restriction $H_n(x)$ by periodic restriction), as they differ by a finite rank perturbation. The function $N(\cdot)$ is monotone and continuous, and its derivative defines a probability measure on $dN(E)$ on $\mathbb R$, which is called the {\it density of states measure}. The topological support of $dN$ is equal to the spectrum of $H(x)$. Note that $\sigma(H(x))$ does not depend on $x$ as a set, see \cite{CFKS}*{Chapters 9,10} and Proposition \ref{lemma_independence}. It is also well known that $dN$ is the average of spectral measures of $H(x)$:
\beq
\label{ids_spectral}
dN(E)=\int_{[0,1)}\langle d\mathbb{E}_{H(x)}(E)e_0,e_0 \rangle\,dx.
\eeq
If one removes the hopping term in \eqref{h_def}, it is well known and immediately follows from ergodic theorems that $N(E)$ becomes $f^{-1}(E)$. If one defines the inverse function $N^{-1}\colon (0,1)\to \R$, it is easy to see from elementary perturbation arguments that
\beq
\label{ids_inv}
|N^{-1}(x)-f(x)|\le \|\Delta\|=2,\quad x\in (0,1).
\eeq

Denote by $M_n(x,E)$ the $n$-step transfer matrix,
\beq
\label{m_def}
M_n(x,E):=\prod_{l=(n-1)}^0\begin{pmatrix}E-f(x+l\alpha)&-1\\1&0 \end{pmatrix}=\begin{pmatrix}P_n(x,E)&-P_{n-1}(x+\alpha,E)\\ P_{n-1}(x,E)&-P_{n-2}(x+\alpha,E)\end{pmatrix},
\eeq
where $P_n(x,E)=\det H_n(x,E)$. Due to $\F_3$, the Lyapunov exponent
\beq
\label{gamma_def1}
L(E)=\lim_{n\to\infty}\frac{1}{n}\int_{[0,1)}\ln\|M_n(x,E)\|\,dx=
\inf_{n\in \N}\frac{1}{n}\int_{[0,1)}\ln\|M_n(x,E)\|\,dx
\eeq
is well defined and finite. It will later be shown that the density of states measure $dN$ has bounded density. Since  \eqref{ids_inv} implies $\log(1+|N^{-1}|)\in L^1(0,1)$, the integral in the right hand side of the Thouless formula
\beq
\label{thouless}
L(E)=\int_{\R} \log|E-E'|\,dN(E').
\eeq
is also well defined, and the formula holds.
\begin{prop}
\label{noac}
For every $x$, $\sigma_{\mathrm{ac}}(H(x))=\varnothing$. As a consequence, $L(E)>0$ for Lebesgue a.~e. $E\in \R$.
\end{prop}
\begin{proof}
The first claim is established in \cite{SS} for all Schr\"odinger operators with unbounded potentials by a deterministic argument. The second claim is a consequence of Kotani theory; the standard references for the bounded case are \cites{CFKS,Kot,SimonKot}, and the extension to the unbounded case is described in \cite{SimonKotUB}*{Appendix II}. See also the remark in the end of \cite{SS}*{Section 1}.
\end{proof}

The following simple lemma immediately follows from \cite{JK}*{Lemma 5.4}. The notation $\gamma^{-1}$ is chosen to match the actual property that will be obtained for $H(x)$ in Theorem \ref{main_ids}.
\begin{lemma}
\label{lyaplip}
Suppose $f\in \F({\gamma})$, and assume that $N(\cdot)$ is Lipschitz continuous: $|N(E)-N(E')|\le \gamma^{-1} |E-E'|$ for all $E,E'\in \R$. Then $L(E)$ is continuous on $\mathbb R$, and
\beq
\label{lyapbound}
L(E)\ge \max\{0,\log(\gamma/2e)\}.
\eeq
\end{lemma}
\begin{proof}
Continuity of $L(E)$ follows from the convergence of the integral \eqref{thouless}. The lower bound follows from the computation in \cite{JK}*{Lemma 5.4}, which is essentially an observation that the worst possible case is $dN(E)=\gamma^{-1}\one_{[-\gamma/2,\gamma/2]}(E)\,dE$.
\end{proof}
A {\it generalized eigenfunction} is a formal solution of the equation $H(x)\psi=E\psi$, satisfying a polynomial bound $\psi(n)\le C(1+|n|)^C$. It is well known that the set of generalized eigenvalues (that is, values of $E$ for which non-trivial generalized eigenfunctions exist) of a self-adjoint Schr\"odinger operator has full spectral measure (see, for example, \cites{Berez1,Simon_semi,Snol,RuiShnol}). The following proposition establishes sufficient conditions to obtain purely point spectrum for an operator family with almost everywhere positive Lyapunov exponent, see also \cite{JK}*{Section 8}.
\begin{prop}
\label{loc_sufficient}
Suppose that the operator family $H(x)$ satisfies the following:
\begin{enumerate}
	\item The density of states measure $dN$ is absolutely continuous with respect to Lebesgue measure.
	\item $L(E)>0$ for Lebesgue a.~e./all $E\in \R$.
	\item For every $(x,E)$ with $L(E)>0$, any generalized eigenfunction $H(x)\psi=E\psi$ belongs to $\ell^2(\Z)$.
\end{enumerate}
Then $\sigma(H(x))$ is purely point for a.~e./all $x\in \R\setminus(\Z+\alpha\Z)$.
\end{prop}
\begin{proof}
Since $dN$ is absolutely continuous, the set of energies for which $L(E)=0$ has density of states measure zero. Hence, due to \eqref{ids_spectral}, it has spectral measure zero for a.~e. $x$. Property 3 implies that, for the remaining full measure set of values of $x$, generalized eigenfunctions of $H(x)$ form a basis in $\ell^2(\Z)$.
\end{proof}
\section{Preliminaries from dynamics: upper bounds on transfer matrices}
Let $(X,\mu,T)$ be a uniquely ergodic dynamical system (for the purposes of this paper, one can assume $X=\T^1$, $Tx=x+\alpha$, $\alpha\in \R\setminus \mathbb Q$, and $\mu$ is the Haar measure). A {\it sub-additive cocycle} on $(X,T,\mu)$ is a family of functions 
$$
h=\{h_n\in L^1(X,\mu)\}_{n\in\mathbb N},\quad h_{n+m}(x)\le h_n(x)+h_m(T^n x).
$$
In this case, one can define the Lyapunov exponent
$$
\Lambda(h)=\lim\limits_{n\to +\infty}\frac{1}{n}\int_X h_n\,d\mu=\inf\limits_{n\to +\infty}\frac{1}{n}\int_X h_n\,d\mu,
$$
which exists due to Kingman's subadditive ergodic theorem. An example of a sub-additive cocycle on $\T^1$ is $\log\|M_n(x,E)\|$ for fixed $E$, and the definition of $\Lambda$ agrees with \eqref{gamma_def1}. The following uniform upper bound \cite{Fu}*{Theorem 1} result will be important (see also \cite{LanaMavi} for a parametric version and some generalizations).
\begin{prop}
\label{furman}
Let $\{h_n\}$ be be a sub-additive cocycle on a uniquely ergodic system $(X,T,\mu)$. Assume, in addition, that $h_n\colon X\to\R$ are continuous and uniformly bounded in $L^1(X,\mu)$. Then, for every $\varepsilon>0$, there exists $N(\varepsilon)$ such that
$$
\frac{1}{n} h_n(x)< \Lambda(h)+\varepsilon,\quad \forall n>N(\varepsilon).
$$
If $h$ and $\Lambda(h)$ continuously depend on a real parameter on a compact interval, then $N(\varepsilon)$ can be chosen uniformly in that parameter.
\end{prop}
One cannot apply Theorem \ref{furman} directly to the transfer matrices \eqref{m_def}. However, one can consider the factorization
\beq
\label{factorization}
M_n(x,E)=F_n(x)G_n(x,E),
\eeq
where
$$
G_n(x,E)=\prod_{l=(n-1)}^0 \frac{1}{1+|f(x+l\alpha)-E|}\begin{pmatrix}E-f(x+l\alpha)&-1\\1&0 \end{pmatrix}
$$
and $F_n(x,E)$ is a scalar function
\beq
\label{eq_F_def}
F_n(x,E)=\prod_{l=0}^{n-1}(1+|f(x+l\alpha)-E|).
\eeq
As long as $\log|f|\in L^1(0,1)$, the Lyapunov exponents of logarithms of all three cocycles are bounded, and, by the additive ergodic theorem
$$
\Lambda(\log\|M\|)=\Lambda(\log|F|)+\Lambda(\log\|G\|)=\Lambda(\log\|G\|)+\int_0^1\log(1+|f(x)-E|)\,dx.
$$
On any finite energy interval, the cocycle $\log\|G\|$ satisfies the assumptions of Proposition \ref{furman}. Hence, while the statement of Proposition \ref{furman} obviously fails for $M$, the only way for this to happen is through $F$ being very large. We will quantify it through introducing a truncation of $f$. Let $B>0$. Define
$$
F^{>B}_n(x,E)=\prod_{l\in [0,n-1]\colon|f(x+l\alpha)-E|> B}(1+|f(x+l\alpha)-E|),
$$
$$
F^{\le B}_n(x,E)=\prod_{l\in [0,n-1]\colon|f(x+l\alpha)-E|\le B}(1+|f(x+l\alpha)-E|),
$$
so that $F_n(x)=F^{>B}_n(x,E)F_n^{\le B}(x,E)$. Let also $l(B,n,x,E)$ be the number of factors in $F_n^{>B}(x,E)$; that is, the number of points of irrational rotation with large values of $f$.
\begin{prop}
\label{prop_discrepancy}
Let $\alpha\in \dcct$ and $B,\varepsilon>0$. Then
\begin{multline}
\label{eq_discrepancy_1}
\l|\frac{1}{n}\log F_n^{\le B}(x,E)-\int\limits_{\{|f(y)-E|\le B\}} \log(1+|f(y)-E|)\,dy\r|\\
\le \frac{C(\alpha,\varepsilon)\log(1+B)}{n^{1/\tau-\varepsilon}}+\frac{2l(B,n,x,E)\log(1+B)}{n}.
\end{multline}
\end{prop}
\begin{proof}
The estimate immediately follows from Koksma's inequality \cite{Kuipers}*{Theorem 2.5.1}, which estimates the difference between an integral and an integral sum through the variation of the integrand (which is the factor $\log(1+B)$) and discrepancy of the sampling set. In this case, we apply the estimate of discrepancy of the irrational rotation \cite{Kuipers}*{Theorem 2.3.2}, which gives the first term in the right hand side. The second term accounts for the fact that $l(B,n,x,E)$ irrational rotation points have been removed from consideration.
\end{proof}

 \begin{lemma}
 \label{lemma_uniform_upper}
 Fix $\alpha\in \dcct$. For $\varepsilon>0$, suppose that the sequence $\{B(n)\}_{n\in \N}$ is chosen to satisfy
 $$
 l(B(n),n,x)\log(1+B(n))=o(n),\quad \log(1+B(n))=o(n^{1/\tau}-\varepsilon),\quad B(n)\to \infty
 $$
 as $n\to \infty$. Then there exists $n_0(\varepsilon)$ such that, for $n>n_0(\varepsilon)$, we have
$$
\log\|M_n(x,E)\|\le n (L(E)+\varepsilon)+\log F_n^{>B(n)}(x,E).
$$
Moreover, $n_0(\varepsilon)$ can be chosen uniformly in $E$ on any compact interval.
\end{lemma}
\begin{proof}
Follows from Proposition \ref{furman} applied to $\log\|G\|$ and Proposition \ref{prop_discrepancy} applied to \eqref{eq_F_def}. The assumptions of the lemma guarantee that both correction terms from \eqref{eq_discrepancy_1} can be absorbed into $\varepsilon$, and the integral over $\{|f(y)-E|\le B\}$ can be replaced by the integral over $(0,1)$.
\end{proof}
\begin{remark}
The general conditions on $B(n)$ are provided for illustrations only. In our applications, it would be sufficient to take $B(n)\approx n$.	
\end{remark}

\section{Box eigenvalues and estimates of the IDS}
The goal of this section is to study the dependence on $x$ of the eigenvalues of $H_n(x)$, and obtain estimates for the IDS, using the definition involving counting functions. The following is one of the main results of this section.
\begin{theorem}
\label{main_ids}
Suppose $f$ satisfies Properties $(\F_1)$ and $(\F_2)$ for some $\gamma>0$. For every $\alpha\in \R\setminus\Q$, the integrated density of states of the operator family $H(x)$ is Lipschitz continuous:
\beq
\label{lip_ids}
|N(E)-N(E')|\le \gamma^{-1} |E-E'|.
\eeq
\end{theorem}
A similar statement is known for a bounded monotone potential, see \cite{JK}*{Theorem 3.1}. The argument in \cite{JK} relies on upper and lower bounds on $f$, and cannot be extended directly to the unbounded case. The proof of Theorem \ref{main_ids} will rely on a different argument and will require several preliminaries. Unlike other results in this paper, it does not rely on Diophantine properties of $\alpha$ or on the strength of the singularity of $f$.
\begin{lemma}
\label{pencils}
Let $A\colon [x_0,x_0+\varepsilon)\to \mnc$ be a continuous self-adjoint matrix-valued function. Suppose $f\colon (x_0,x_0+\varepsilon)\to \R$ is continuous, and $f(x_0+0)=\pm\infty$. Let $P$ be a rank one orthogonal projection. Then
$$
\lim\limits_{x\to x_0+}\sigma(A(x)+f(x)P)=\sigma\l(\l.(1-P)A(x_0)\r|_{\ran(1-P)}\r)\cup\{\pm\infty\}.
$$
where the convergence is understood in the sense of sets with multiplicities.
\end{lemma}
\begin{proof}
The $\pm\infty$ part follows from elementary rank one perturbation theory. There is exactly one eigenvalue escaping to $\pm\infty$, and the other ones remain bounded. Let $P=\langle \cdot,e\rangle e$. One can rewrite
$$
A(x)+f(x) P=\begin{pmatrix}
	f(x)+f_0(x)& c(x)^T\\
	c(x) & B(x)
\end{pmatrix},\quad B(x)=\l.(1-P)A(x)\r|_{\ran(1-P)},
$$
$$
c(x)=(1-P)A(x)e,\quad f_0(x)=\<A(x)e,e\>.
$$
Let also $v$ be an eigenvector of $B(x_0)=\l.(1-P)A(x_0)\r|_{\ran(1-P)}$, $B(x_0)v=\mu v$. For $x$ sufficiently close to $x_0$, let 
$$
\varepsilon(x)=-\frac{c(x)^T v}{f(x)+f_0(x)}=O(|f(x)|^{-1})=o(1).
$$ 
Then
$$
(A(x)+f(x) P)\begin{pmatrix}
	\varepsilon(x)\\
	v
\end{pmatrix}=\begin{pmatrix}
	f(x)+f_0(x)& c(x)^T\\
	c(x) & B(x)
\end{pmatrix}
\begin{pmatrix}
	\varepsilon(x)\\
	v
\end{pmatrix}
$$
$$
=\begin{pmatrix}
	0\\
	\varepsilon(x)c(x)+\mu v+(B(x)-B(x_0))v
\end{pmatrix}=
\mu \begin{pmatrix}
	\varepsilon(x)\\
	v
\end{pmatrix}
+o(1).
$$
This implies that $\dist(\mu,\sigma(A(x)+f(x) P))=o(1)$, and hence $\mu$ is a limit point of $\sigma(A(x)+f(x) P)$. Without loss of generality, one can assume that the eigenvalues of $B(x_0)$ are distinct, which completes the proof (alternatively, one can rewrite the argument using spectral projections).
\end{proof}
Fix some $n\in \N$, and denote by $E_j(x)$, $0\le j\le n-1$, the $j$-th eigenvalue of $H_n(x)$ in the increasing order. Let us denote the discontinuity points of the diagonal entries of $H_n(x)$ by
\beq
\label{beta_def}
\{\beta_0,\ldots,\beta_{n-1}\}=\{\{-j\alpha\},\,0\le j\le n-1\},\quad 0=\beta_0<\beta_1\ldots<\beta_{n-1}<\beta_n=1.
\eeq
One can identify $\beta_0$ with $\beta_n$, as all functions are $1$-periodic; however, the notation \eqref{beta_def} will still be convenient.
\begin{lemma}
\label{jumps}
The functions $E_j(x)$ are continuous and Lipschitz monotone on the intervals $(\beta_l,\beta_{l+1})$, $0\le l\le n-1$. Moreover, 
\beq
\label{jumps1}
E_0(\beta_l+0)=-\infty,\,\, E_{n-1}(\beta_l-0)=+\infty, \quad 0\le l\le n;
\eeq
\beq
\label{jumps2}
E_j(\beta_l-0)=E_{j+1}(\beta_l+0),\quad 1\le j\le n-2, \quad 0\le l\le n.
\eeq
\end{lemma}
\begin{proof}
All claims follow from Lemma \ref{pencils}. At each point of discontinuity, there is a matrix element $f(x-\beta_l)$ that approaches $\pm\infty$, and the other matrix elements remain continuous and bounded. The re-numbering in \eqref{jumps2} is caused by one eigenvalue changing from very large negative to very large positive.
\end{proof}
As a consequence, all discontinuities in $E_j(x)$ are either caused by an eigenvalue diverging to infinity, or can be fixed by suitable re-numbering. We can define {\it continuous eigenvalue curves} by
\beq
\label{lambdaj_def}
\Lambda_j(x)=E_{(j+l)\,\mathrm{mod}\, n}(x),\quad x\in (\beta_l,\beta_{l+1}), \quad 0\le j,l\le n-1.
\eeq
\begin{cor}
\label{Sinai1}
\begin{enumerate}
	\item The functions $\Lambda_j$ can be extended into $\R\setminus\{\Z+\beta_{n-j}\}$ by continuity and $1$-periodicity.
	\item $\Lambda_j$ are Lipschitz monotone on $(\beta_{n-j},\beta_{n-j}+1)$, and $\Lambda_j(\cdot-\beta_{n-j})\in \F(\gamma)$. Moreover, $\log(1+|\Lambda_j(\cdot-\beta_{n-j})|)$ belong to $L^1(0,1)$ uniformly in $n$ and $j$.
	\item $|\Lambda_j(x)-f(x-\beta_{n-j})|\le 2$ for all $x\in \R\setminus\{\Z+\beta_{n-j}\}$.
\end{enumerate}
\end{cor}
\begin{proof}
Part $(1)$ follows from \eqref{jumps2}. Part $(2)$ follows from elementary perturbation theory. Part $(3)$ follows from comparing the eigenvalues of $H$ and the diagonal part of $H$, and the fact that $\|\Delta\|\le 2$.
\end{proof}

{\noindent \it Proof of Theorem $\ref{main_ids}$. }Recall the definition \eqref{ids_def}. Let $E_1<E_2$.
$$
N(E_2)-N(E_1)=\lim\limits_{n\to\infty}\frac{1}{n}\int_{[0,1)}\sum_{j=0}^{n-1}\one_{(E_1,E_2]}(E_j(x))=\lim\limits_{n\to\infty}\frac{1}{n}\sum_{j=0}^{n-1}\int_{[0,1)}\one_{(E_1,E_2]}(\Lambda_j(x))
$$
$$
\le \frac1n \sum_{j=0}^n \gamma^{-1}(E_2-E_1)=\gamma^{-1}(E_2-E_1),
$$
where the inequality follows from Lipschitz monotonicity property of $\Lambda_j(\cdot)$.\,\qed
\begin{remark}
In the setting of \cite{JK} (i.~e. when $f$ is bounded), one cannot, in general, choose continuous eigenvalue curves, as the jumps at finite energies will affect all eigenvalues. However, we will always have an inequality $E_j(\beta_l-0)\le E_{j+1}(\beta_l+0)$ under the same assumptions as \eqref{jumps2}, and the $\Lambda_j$ defined by \eqref{lambdaj_def} will have only positive jump discontinuities, except at $\beta_{n-j}$. Hence, the method from this section is also applicable there, and actually gives stronger result.
\end{remark}
\section{Box eigenvalue distribution modulo finite rank}

In this section, we establish some more delicate properties of the eigenvalue distribution of $H_q(x)$, for some special values of $q$. Suppose, we are only interested in the eigenvalue counting function $N_q(x,E)$ modulo some absolute constant, as $q\to \infty$; or, equivalently, all finite rank perturbations, as long as the total perturbation rank is bounded. The main conclusion of this section is that, modulo those assumptions, the separation properties of the eigenvalues $E_j(x)$, $0\le j\le q-1$, are at least as good as those of the sequence $\{\Lambda_0(x),\Lambda_0(x+\alpha),\ldots,\Lambda_0(x+(q-1)\alpha)\}$.

We will always denote $\Lambda(x)=\Lambda_0(x)$ (however, it will be clear later that one can use any $\Lambda_j$ for that purpose); the dimension of the box under consideration will be clear from the context. Whenever we write $H_q(x)$, we assume that $q$ is a {\it good denominator} of $\alpha$; that is, the points $\{j\alpha,1\le j\le q-1\}$, split $(0,1)$ into $q$ intervals of lengths between $\frac{1}{2q}$ and $\frac{3}{2q}$. WLOG, we can also assume that $\|q\alpha\|\le \frac{1}{\sqrt{5}q}$, and that each interval $\l(j/q,(j+1)/q\r)$ contains exactly one point of the form $\{\alpha\},\{2\alpha\}\ldots,\{q\alpha\}$.

It will be convenient to use the following notation: for two $(n\times n)$-matrices $A,B$, we say $A\approx B$ if $A$ is unitarily equivalent to $B$ modulo a finite rank perturbation; here, ``finite'' means ``bounded by an absolute constant''. Similarly, we will use the notation $A\lesssim B$ and $A\gtrsim B$.
\begin{lemma}
\label{invariance1}
Suppose $1\le n\le q$, $0\le y\le 2/q$. Then
$$
H_n(x)\lesssim H_n(x+y).
$$
\end{lemma}
\begin{proof}
On the interval $[x,x+y]$, all diagonal entries of $H_n(\cdot)$ increase monotonically, except for at most three jump points. Each jump generates at most rank one perturbation.
\end{proof}
\begin{lemma}
\label{invariance2}
For $0\le m\le q-1$, we have
$$
H_q(x)\lesssim H_q(x+m\alpha)\lesssim H_q(x+q\alpha),\quad \textrm{if}\quad \{q\alpha\}<1/2;
$$
$$
H_q(x+q_k\alpha)\lesssim H_q(x+m\alpha)\lesssim H_q(x),\quad \textrm{if}\quad \{q\alpha\}>1/2.
$$
\end{lemma}
\begin{proof}
Note that $\{q\alpha\}$ is either very close (closer than $1/\sqrt{5}q$) to $0$, or to $1$. In the first case, we have
\beq
\label{invariance_eq1}
H_q(x+m\alpha)\approx H_m(x+q\alpha)\oplus H_{q-m}(x+m\alpha)\gtrsim H_m(x)\oplus H_{q-m}(x+m\alpha)\approx H_q(x).
\eeq
In the first equality we applied cyclic permutation of basis vectors, then a finite rank perturbation to transform the off-diagonal part into the usual Laplacian, and then removed two off-diagonal entries so that the operator de-couples; the total is a rank $4$ perturbation after unitary equivalence. The inequality follows from Lemma \ref{invariance1}.

Now, we can apply \eqref{invariance_eq1} again, replacing $x$ by $x+m\alpha$ and $m$ by $q-m$, and obtain
$$
H_q(x+q\alpha)=H_q(x+m\alpha+(q-m)\alpha)\gtrsim H_q(x+m\alpha)\gtrsim H_q(x).
$$
The case $\{q\alpha\}>1/2$ is similar, if one reverses the inequalities. One can check that in both cases, each inequality is at most rank $6$.
\end{proof}
\begin{cor}
\label{invariance3}
For every $x\in [0,1)$, $E\in \R$, $1\le|m|\le q-1$, there exists $y\in [0,1)$, $\|x-y\|\le \|q\alpha\|$ such that
$$
|N_q(x,E)-N_q(y+m\alpha,E)|\le C,
$$
where $C$ is an absolute constant.
\end{cor}
\begin{proof}
The claim immediately follows from Lemma \ref{invariance2}. The case $m<0$ can be obtained replacing $x$ by $x-m\alpha$. One can check that $C=30$ is sufficient.
\end{proof}
\begin{cor}
\label{constant1}
There exists an absolute constant $C$ such that
$$
|N_q(x,E)-N_q(y,E)|\le C,\quad \forall x,y\in[0,1).
$$
\end{cor}
\begin{proof}
Due to monotonicity of $\Lambda_j$, as $x$ runs over $[0,1)$, the counting function $N_q(x,E)$ changes its values $2q$ times. That is, the value is decreased by $1$ whenever $\Lambda_j(x)=E$, and increased by $1$ at the points $\beta_l$. Due to Corollary \ref{invariance3}, it would be sufficient to show that $N(\cdot,E)$ cannot decrease more than by $C_1$ on any interval $(\beta_l,\beta_{l+1})$. Suppose that it does. Then there exists an interval of size $\frac{1}{10q}$ such that the function decreases by, say, $C_1/20$ on that interval. Corollary \ref{invariance3} implies that it will decrease by $C_1/20-C$ on at least $q/3$ non-overlapping intervals of size $\frac{1}{10q}$. For sufficiently large $C_1$, this contradicts the fact that the total increment on any subset of $[0,1)$ is at most $q$.
\end{proof}
\begin{lemma}
\label{invariance4}
Let $N_q(x,E)$ be the eigenvalue counting function of $H_q(x)$, and $N_q^{\Lambda}(x,E)$ be the counting function of the set 
\beq
\{\Lambda(x+j\alpha),\,0\le j\le q-1\}.
\eeq
Then
\beq
\label{invariance4_eq1}
|N_q^{\Lambda}(x,E)-N_q(y,E)|\le C,\quad \forall x,y\in [0,1),
\eeq
where $C$ is an absolute constant.
\end{lemma}
\begin{proof}
Since $q$ is a good denominator, the values of $N^{\Lambda}_q(x,E)$ do not change more than by a constant (in fact, more than by $1$) as $x$ runs over $[0,1)$. Moreover, the definition of $\Lambda$ implies that
$$
\{\Lambda(x+j\alpha),\,0\le j\le q-1\}=\{E_0(x_0),E_1(x_1),\ldots,E_{q-1}(x_{q-1})\}
$$
for some collection of points $x_0,\ldots,x_{q-1}$. Since the counting function of the eigenvalues does not depend on $x$ modulo constant, the choice of different arguments $x_j$ for different eigenvalues will not change the counting function more than by a constant.
\end{proof}

\section{A large deviation theorem}
Let $P_n(x,E)=\det(H_n(x)-E)$. Typically, a large deviation theorem states that $P_n(x,E)=n(L(E)+o(1))$, except for some ``small'' set of values of $x$. In the present situation, we do not have a uniform upper bound on $P$, and we will have some correction term that can be very large. The goal of this section is to isolate the contribution of the correction term to the LDT, and show that it will later cancel in the Green's function estimate. The following lemma is elementary and was also used in \cite{JK}.
\begin{lemma}
\label{goodlemma}
Suppose that $A_1$, $A_2$ are two finite subsets of $[m,M]$ of the same cardinality, 
$m>0$, and that $f$ is a nondecreasing function on $[m,M]$.
Assume that the difference of counting functions of $A_1$ and $A_2$ is bounded by $N$. Then
$$
\l|\sum_{a\in A_1}f(a)-\sum_{a\in A_2}f(a)\r|\le 2N\max\{|f(m)|, |f(M)|\}.
$$
\end{lemma}
Recall
$$
F^{>B}_n(x,E)=\prod_{l\in [0,n-1]\colon|f(x+l\alpha)-E|> B}(1+|f(x+l\alpha)-E|),
$$

\begin{theorem}
\label{th_lower_ldt}
Suppose that the sequence $\{B(n)\}_{n\in \mathbb N}$ satisfies the assumptions of Lemma $\ref{lemma_uniform_upper}$. Let $k$ deliver the smallest value:
$$
|E_{k}(x)-E|=\min\limits_{0\le k\le q-1}|E_k(x)-E|.
$$
For every $\varepsilon>0$ there exists $n_0(\varepsilon)$ such that, if $q>n_0(\varepsilon)$ is a good denominator, then
\beq
\label{lower_eq}
\log|P_q(x,E)|\ge q\l(\int_0^1\log|\Lambda(x)-E|\,dx-\varepsilon\r)+\log F_q^{>B(q)}(x,E)+C\log|E_{k}(x)-E|.
\eeq
The estimates can be made uniform in $E$ on any compact interval.
\end{theorem}
\begin{proof}
By definition,
$$
P_q(x,E)=\prod_{k=0}^{q-1}(E_k(x)-E).
$$
Corollary \ref{constant1} implies that on any interval $(\beta_l,\beta_{l+1})$, there are at most $C$ values of $k$ with $E_k(x)=E$ with $x$ from that interval. Hence, there are at most $C$ values of $k$ with $|E_k(x)-E|\le 1/q$. The contribution of these eigenvalues is contained in the last term of \eqref{lower_eq}.

The second term of \eqref{lower_eq} is a lower bound on the contribution from the eigenvalues with $|E_j(x)-E|\ge B(q)+3$, see Corollary \ref{Sinai1} (3). Let us call the remaining eigenvalues {\it regular}. Clearly, regular eigenvalues satisfy
$$
q^{-1}\le |E_j(x)-E|\le B(q)+3.
$$
The assumptions of Lemma \ref{lemma_uniform_upper} imply that there are $o(q)$ non-regular eigenvalues, uniformly in $x$. Denote
$$
\tilde\Lambda(x)=\begin{cases}
	\Lambda(x),& q^{-1}\le |\Lambda(x)-E|\le B(q)+3\\
	E+1,& \text{otherwise}.
\end{cases}
$$
From Lemmas \ref{invariance4} and \ref{goodlemma}, it follows that one can replace $E_j(x)$ in the contribution to \eqref{lower_eq} from regular eigenvalues by $\Lambda(x+j\alpha)$, with an error that can be absorbed into $\varepsilon$:
$$
\sum\limits_{k\colon E_k \text{ is regular}}\log|E-E_k(x)|=\sum\limits_{k=0}^{q-1}\log|\tilde\Lambda(x+j\alpha)-E|+o(q).
$$
Using Koksma's inequality similarly to \eqref{eq_discrepancy_1}, we can see that
$$
\sum\limits_{k\colon E_k \text{ is regular}}\log|E_k(x)-E|=q\int_0^1\log|\tilde \Lambda(x)-E|\,dx+o(q)=q\int_0^1\log|\Lambda(x)-E|\,dx+o(q),
$$
where we used uniform absolute continuity of the integral (Corollary \ref{Sinai1}) to remove the truncation from $\tilde\Lambda$.
\end{proof}
\begin{lemma}
\label{lemma_lambda_lyap}
In the notation of the previous theorem,
$$
\frac{1}{q}\int_0^1 \log|P_q(x,E)|\,dx=L(E)+o(1).
$$
\end{lemma}
\begin{proof}
Integrating the inequality from Lemma \ref{lemma_uniform_upper} and using \eqref{m_def}, we can see that
$$
\frac{1}{n}\int_0^1 \log|P_n(x,E)|+o(1)\le L(E)\le 
$$
$$
\le \frac{1}{n}\int_0^1\log\l(|P_{n-2}(x,E)|+|P_{n-1}(x,E)|+|P_{n}(x,E)|\r)+o(1),
$$
assuming $|n-q|\le C$ for some good denominator $q$. On the other hand, Lemma \ref{goodlemma} implies that, under the same assumption,
$$
\frac{1}{n}\int_0^1 \log|P_n(x,E)|=\frac{1}{q}\int_0^1 \log|P_q(x,E)|+o(1)
$$
(we apply Lemma \ref{goodlemma} to regular eigenvalues, and note that a finite rank perturbation would create at most finitely many singular eigenvalues, whose contribution will be $o(1)$).
\end{proof}
\begin{theorem}
\label{th_green_bound}
Under the assumptions and notation of Theorem $\ref{th_lower_ldt}$, let $\varepsilon,\sigma>0$. Suppose $0\le l_1<l_2\le q-1$, $l_2=l_1+n$, and $n\ge \sigma q$, where $q>n_0(\varepsilon,\delta)$ is a good denominator. Then
\beq
\label{eq_ratio_det}
\l|\frac{P_{n}(x+l_1\alpha)}{P_q(x)}\r|\le e^{-(q-n)(L(E)-\varepsilon)}|E_k(x)-E|^{-C}.
\eeq
\end{theorem}
\begin{proof}
Lemma \ref{lemma_uniform_upper} implies an upper bound on the numerator:
$$
\log |P_{l_2-l_1+1}(x+l_1\alpha)|\le n (L(E)+\varepsilon_1)+\log F_n^{>B_1(n)}(x,E).
$$
Theorem \ref{th_lower_ldt} and Lemma \ref{lemma_lambda_lyap} imply the lower bound on the denominator
$$
\log|P_q(x,E)|\ge q\l(L(E)-\varepsilon_2\r)+\log F_q^{>B_2(q)}(x,E)+C\log|E_{k}(x)-E|,
$$
where both $B_1$ and $B_2$ satisfy the assumptions of Lemma \ref{lemma_uniform_upper}. Since $n\ge \delta q$ and any sequence $B(n)=cn$ satisfies the assumptions, one can pick, say, $B_1(n)=n$, $B_2(n)=\sigma n$, which would imply
$$
\log F_n^{>B_1(n)}(x,E)\le \log F_q^{>B_2(q)}(x,E).
$$
The latter implies that all terms that can possibly violate uniform upper bound in the numerator, will be cancelled by similar terms in the denominator, and the upper bound on the ratio will have the correct form \eqref{eq_ratio_det}, after an appropriate choice of $\varepsilon$.
\end{proof}
\section{Proof of Theorem \ref{main}}
Suppose that $\psi\colon \Z\to \C$ is a non-trivial solution of
$$
H(x)\psi=E\psi,\quad |\psi(n)|\le C(1+|n|)^C.
$$
Recall that $\psi$ is called a {\it generalized eigenfunction} of $H(x)$, and $E$ is called a {\it generalized eigenvalue} of $H(x)$. In this section, we will prove the following theorem, which implies the main result (see Proposition \ref{loc_sufficient} and Theorem \ref{main}).
\begin{theorem}
\label{main_decay}
Suppose $\alpha\in \dc$, $f\in \F({\gamma})$. Suppose that $x\notin \Z+\alpha\Z$, and $E$ is a generalized eigenvalue of $H(x)$ with $L(E)>0$. Then the corresponding generalized eigenfunction $\psi$ belongs to $\ell^2(\Z)$.
\end{theorem}
The method of proof is close to \cite{J_annals} and \cite{JK}. In fact, the line of the argument is very close to \cite{JK} with some modifications in applying large deviations.
We will need some preliminaries. For an interval $[a,b]\subset \Z$, denote
$$
G_{[a,b]}(x;m,n)=(H_{[a,b]}(x)-E)^{-1}(m,n)=(H_{b-a}(x+a)-E)^{-1}(m,n).
$$
A point $m\in \Z$ is called $(\mu,q)$-regular, if there is an interval $[n_1,n_2]=[n_1,n_1+q-1]$, $m\in [n_1,n_2]$, $|m-n_i|\ge q/5$, and
$$
|G_{[n_1,n_2]}(x;m,n_i)|<e^{-\mu|m-n_i|}.
$$
The Poisson formula
$$
\psi(m)=-G_{[n_1,n_2]}(x;m,n_1)\psi(n_1-1)-G_{[n_1,n_2]}(x;m,n_2)\psi(n_2+1),\quad m\in [n_1,n_2].
$$
implies that any point $m$ with $\psi(m)\neq 0$ is $(\mu,q)$-singular for sufficiently large $q$.
\begin{lemma}
\label{multiscale}
Under the assumptions of Theorem $\ref{main_decay}$, fix $0<\delta<L(E)$. There exists $q_0(f,\alpha,\delta)$ such that, for any good denominator $q>q_0$, any two $(L(E)-\delta,q)$-singular points $m,n$ with $|m-n|>\frac{q+1}{2}$, satisfy $|m-n|>e^{C(\alpha,f)q}$.
\end{lemma}
\begin{proof}
The proof follows exactly the scheme from \cite{JK} which, in turn, is partially based on \cite{J_annals}.
We have the following expressions for Green's function matrix elements if $b=a+q-1$, $a\le l\le b$.
\beq
\label{greenleft}
|G_{[a,b]}(x;a,l)|=\l|\frac{P_{b-l}(x+(l+1)\alpha)}{P_{q}(x+a\alpha)}\r|,
\eeq
\beq
\label{greenright}
|G_{[a,b]}(x;l,b)|=\l|\frac{P_{l-a}(x+a \alpha)}{P_{q}(x+a\alpha)}\r|.
\eeq
Suppose that $m-[3q/4]\le l\le m-[3q/4]+[(q+1)/2]$ and $m$ is $(\gamma(E)-\delta,q)$-singular. Then
\beq
\label{greenfunctionlarge}
|G_{[a,b]}(x;a,l)|>e^{-(l-a)(\gamma(E)-\delta)}\quad \text{or}\quad  |G_{[a,b]}(x;l,b)|>e^{-(b-l)(\gamma(E)-\delta)}
\eeq
for all intervals $[a,b]$ such that $|a-l|,|b-l|\ge q/5$ and $b=a+q-1$. However, Theorem \ref{th_green_bound} with $n\ge q/5$ implies
$$
|G_{[a,b]}(x;a,l)|\le e^{-(l-a)(\gamma(E)-\varepsilon)}|E_k(x+a\alpha)-E|^{-C},
$$
$$
|G_{[a,b]}(x;l,b)|\le e^{-(b-l)(\gamma(E)-\varepsilon)}|E_k(x+a\alpha)-E|^{-C}
$$
for any fixed $\varepsilon>0$ and large $q$. Take $\varepsilon<\delta/2$. Then, the only way to obtain \eqref{greenfunctionlarge} is to have
$$
|E_k(x+a\alpha)-E|\le e^{-q\delta/10C},\quad\text{that is},\quad \dist(x,Z_q)\le \gamma^{-1}e^{-q\delta/10C},
$$
where $Z_q=\{x\colon \det (H_q(x)-E)=0\}$ is a set of cardinality $\le q$ and the bound follows from Lipschitz behavior of eigenvalues.

Suppose now that the points $m_1$ and $m_2=m_1+r$
are both $(L(E)-\delta,q)$-singular, $r>0$. Let 
$$
x_j=\{x+(m_1-[3q/4]+(q-1)/2+j)\alpha\},\quad j=0,\ldots, [(q+1)/2]-1,
$$
$$
x_j=\{x+(m_2-[3q/4]+(q-1)/2+j-[(q_k+1)/2])\alpha\},\quad j=[(q+1)/2],\ldots, q.
$$
If $r>\frac{q+1}{2}$, then all these points are distinct and must be exponentially close to $Z_q$. As a consequence, two of the points $x_j$ must be exponentially close to each other, which is only possible, due to the Diophantine condition, if $m_1$ is exponentially separated from $m_2$.
\end{proof}
Theorem \ref{main_decay} now follows in the same way as in \cite{JK}, as all intervals of the form $[q,e^{Cq})$ become $q$-regular for all sufficiently large good denominators $q$, the intervals $[q,e^{Cq})$ cover $\Z_+$ except for finitely many points, and similar arguments can be applied on $\Z_-$.
\begin{remark}
\label{lyaploc}
Similarly to \cite{JK} (and with the same proof), the operator family $H(x)$ has {\it uniform Lyapunov localization}. That is, on any compact energy interval and for any $\delta>0$, there exists $C(\delta)$ such that for any
eigenfunction $\psi$  satisfying $H \psi=E\psi$, $\|\psi\|_{l^{\infty}}=1$, 
there exists $n_0(\psi)$ so that we have
$$
|\psi(n)|\le C(\delta)e^{-(L(E)-\delta)|n-n_0(\psi)|}.
$$
In particular, we have uniform localization on any compact interval of uniform positivity of $L(E)$.
\end{remark}
\section{Proof of Theorem \ref{th_spec_R}}
\noindent Let $\Theta(\alpha)=[0,1)\setminus(\Z+\alpha\Z)$. Recall that $\{x\}$ denotes the fractional part of $x$.
\begin{prop}
\label{lemma_independence}
Let $\alpha\in \R\setminus\Q$, and $x,y\in \Theta(\alpha)$. Then $\sigma(H(x))=\sigma(H(y))$.
\end{prop}
\begin{proof}
A similar argument was used in \cite[Lemma 4.1]{JL1}. 
Let $m_j$ be a sequence of integers such that $\{x+m_j\alpha\}\to y$, $j\to \infty$. Clearly, for any $\psi$ with finite support,
\beq
\label{eq_strong_resolvent}
H(x+m_j\alpha)\psi\to H(y)\psi.
\eeq
Since all operators $H(x)$, $x\in \Theta(\alpha)$, are essentially self-adjoint on the set of vectors with finite support, \eqref{eq_strong_resolvent} and \cite[Theorem VIII.25]{rs1} imply that $H(x+m_j\alpha)$ converges to $H(y)$ in the strong resolvent sense. Since all operators in the left hand side of \eqref{eq_strong_resolvent} are unitary equivalent, it follows from \cite[Theorem VIII.24]{rs1} that $\sigma(H(y))\subset\sigma(H(x))$ (strong resolvent convergence cannot create new spectra). As $x$, $y$ are arbitrary, this completes the proof.
\end{proof}
Let
\beq
\label{eq_union_spectrum}
\Sigma_+(\alpha)=\bigcup_{x\in \Theta(\alpha)}\sigma(H(x))
\eeq
be the {\it union spectrum} of the family $H(x)$. Note that the union is tautological for $\alpha\in \R\setminus\Q$, but is not trivial for $\alpha\in \Q$.
\begin{lemma}
\label{lemma_rational_spectrum}
Let $\alpha\in \Q$. Then $\Sigma_+(\alpha)=\R$.	
\end{lemma}
\begin{proof}
We will prove a stronger statement is true: let $\alpha=p/q$ where $(p,q)=1$, and apply the Floquet decomposition to the periodic operator $H(x)$:
\beq
\label{eq_floquet}
H(x)=\int_{[0,1/q)}\oplus h_{\theta}(x)\,d\theta.
\eeq
Then, for each fixed $\theta$,
\beq
\label{eq_spec_theta_R}
\bigcup_{x\in (0,1/q)}\sigma(h_{\theta}(x))=\R.
\eeq 
It is well known that $h_{\theta}(x+1/q)$ is unitary equivalent to $h_{\theta}(x)$. Denote by $E_j(x)$, $j=0,\ldots,q-1$, the eigenvalues of $h_{\theta}(x)$ counted with multiplicities, assuming that $x\in (0,1)\setminus\{1/q,2/q,\ldots,(q-1)/q\}$. Clearly, $E_0(0+)=-\infty$ and $E_0(1/q-)$ is finite. Due to Lemma \ref{pencils}, the graph of $E_0(x)$ must extend continuously through $x=1/q$. Since $E_0(1/q+)=E_0(0+)=-\infty$, we have $E_0(1/q-)=E_1(1/q+)=E_1(0+)$. Similarly, $E_2(0+)=E_1(1/q-)$, etc.
By repeating this process, we show that the closure of the ranges of $E_j$ on $(0,1/q)$ covers $\mathbb R$. Alternatively, one can repeat the proof of Corollary \ref{Sinai1} and notice that, in the rational case, the curves defined by \eqref{lambdaj_def} are the translations of the same curve by multiples of $1/q$. If $f$ is monotone, one can additionally conclude that the covering has multiplicity one.
\end{proof}
The proof of the following theorem goes along the same lines as the arguments in \cite[Theorem 3.6]{avron_simon_ids} and \cite[Lemma 2.3]{bellissard_simon_cantor}. The technical difficulty in the unbounded case is that one has to deal with sequences of phases that approach the ``resonant set'' $\Z+\alpha\Z$, for which the operator is not well defined.

{\noindent \bf Proof of Theorem 1.2. }
Due to Lemma \ref{lemma_rational_spectrum}, one can consider the case $\alpha\in \R\setminus\Q$, in which $\Sigma:=\sigma(H(x))$ does not depend on $x\in \Theta(\alpha)$. By contradiction, assume that $I\subset\R\setminus\Sigma$ is a non-empty open interval. Let
$$
P_j=\langle\cdot,e_j\rangle e_j
$$
be the projection on the basis vector in $\ell^2(\Z)$, and $P_j^{\perp}=I-P_j$. Take any sequence $x_j\in \Theta(\alpha)$, $x_j\to 0$. Then
\beq
\label{eq_p0perp}
P_0^{\perp}H(x_j)P_0^{\perp}\to P_0^{\perp} H(0)P_0^{\perp},
\eeq
where in $H(0)$ we replace the infinite value of the potential at $n=0$ by zero, and the convergence is considered in the strong resolvent sense (can be verified on vectors with finite support using \cite[Theorem VIII.25]{rs1}).

Each operator in the left hand side of \eqref{eq_p0perp} is at most rank two perturbation of $H(x_j)$, and hence has at most two points of spectrum in $I$. By passing to a subsequence, one can assume that
$$
\sigma(P_0^{\perp}H(x_j)P_0^{\perp})\cap I\to \sigma_0,\quad j\to \infty
$$
where $\#\sigma_0\le 2$, which implies that
$$
\sigma(P_0^{\perp} H(0)P_0^{\perp})\cap I\subset \sigma_0,
$$
and hence
$$
\|(P_0^{\perp} H(0)P_0^{\perp}-E)\psi\|\ge \dist (E,\sigma_0)\|\psi\|,\quad E\in I,
$$
for every $\psi\in \ell^2(\Z)$ with finite support. Similarly, one can define $P_k^{\perp} H(k\alpha) P_k^{\perp}$ by translation, which will all be unitarily equivalent for different values of $k$, and obtain
\beq
\label{eq_pk_resolvent}
\|(P_k^{\perp} H(k\alpha)P_k^{\perp}-E)\psi\|\ge \dist (E,\sigma_0)\|\psi\|,\quad E\in I,
\eeq
also for $\psi\in \ell^2$ with finite support.

Take $E\in I\setminus\sigma_0$, and take any sequence $p_n/q_n\to\alpha$ as $n\to \infty$. Similarly to \cite[Lemma 2.3]{bellissard_simon_cantor}, pick
\beq
\label{eq_xn_nonres}
x_n\in [0,1)\setminus \Theta(p_n/q_n)
\eeq
and a generalized eigenfunction $\psi_n$ such that
\beq
\label{eq_psik_eigen}
H(p_n/q_n)\psi_n =E\psi_n, \quad \psi_n(0)=1/2,\quad \|\psi_n\|_{\infty}\le 1.
\eeq
Note that the second condition can be achieved by translation, and \eqref{eq_xn_nonres} can be satisfied by perturbing the quasimomentum if needed. By passing to a subsequence and applying the diagonal process, one can assume that $x_n$ has a limit $x\in \R/\Z$, and $\psi_n$ has a pointwise limit 
$$
\psi\in \ell^{\infty}(\Z),\quad \psi(0)=1/2.
$$
If $x\in \Theta(\alpha)$, this completes the proof, as $\psi$ satisfies the eigenvalue equation $H(x)\psi=E\psi$, which contradicts the fact that $E\notin \Sigma$ and Schnol's theorem. Assume now that $x=\{k\alpha\}$, $k\in \Z$. Since $|f(x_n+k p_n/q_n)|\to \infty$, the eigenvalue equation for $\psi_n$ implies that $\psi_n(k)\to 0$ as $n\to \infty$, and hence $k\neq 0$. Similarly to the previous arguments, one can check that
$$
P_k^{\perp}H(x_n+k p_n/q_n)P_k^{\perp}\to P_k^{\perp}H(k\alpha)P_k^{\perp},\quad n\to \infty,
$$
in the strong resolvent sense. Moreover, by applying $P_k^{\perp}$ to \eqref{eq_psik_eigen}, one can see that $\psi$ satisfies the eigenvalue equation
\beq
\label{eq_psi_eigen}
P_k^{\perp}H(k\alpha)P_k^{\perp}\psi=E\psi;
\eeq
note that it is not necessarily true without $P_k^{\perp}$, as $f(x_n+k p_n/q_n)\psi_n(k)$ may not converge to zero.

The equation \eqref{eq_psi_eigen} contradicts \eqref{eq_pk_resolvent}, if one applies it to the truncated vector $\one_{[-m,m]}\psi$. Indeed, if $\psi\in \ell^2$, then $\psi$ is in the domain of $P_k^{\perp}H(k\alpha)P_k^{\perp}$, and hence the truncation error in the left hand side of \eqref{eq_psi_eigen} goes to zero. If $\psi\notin \ell^2(\Z)$, then one would eventually arrive to the same contradiction by choosing a sequence $m_j$ with $|f(x+m_j\alpha)|\le C$.
\section{Acknowledgements} The research was supported by the DMS--1600422/DMS--1758326 ``Spectral Theory of Periodic and Quasiperiodic Quantum Systems''. 
The author would like to thank Svetlana Jitomirskaya, Wencai Liu, Jeffrey Schenker, and Shiwen Zhang for valuable discussions and advice.
\bibliography{mary4} 
\bibliographystyle{plain}
\end{document}